\documentclass[12pt,a4paper]{article}
\usepackage{amsfonts, amsthm, amsmath, amssymb}

  \newtheorem{theorem}{Theorem}[section]
  \newtheorem{lemma}[theorem]{Lemma}
  \newtheorem{proposition}[theorem]{Proposition}
  \newtheorem{corollary}[theorem]{Corollary}
  
  \newtheorem {example}[theorem] {Example}
  \newtheorem{remark}[theorem]{Remark}
  \numberwithin{equation}{section}
\begin{document}
\baselineskip=17pt
  \begin{center}
  {\large\bf On the spectrum of the operator which is a composition of integration and
  substitution}\\
  by\\
  Ignat Domanov (Donetsk)
\end{center}
{\small{\bf Abstract.} Let $\phi :  [0,1]\longrightarrow [0,1]$ be a
nondecreasing continuous function such that $\phi(x)>x$ for all
$x\in (0,1)$. Let the operator $V_{\phi} :\
f(x)\rightarrow\int\limits_0^{\phi(x)}f(t)dt$ be defined on
$L_2[0,1]$. We  prove that $V_{\phi}$ has a finite number of
non-zero eigenvalues  if and only if  $\phi(0)>0$ and
$\phi(1-\varepsilon)=1$ for some $0<\varepsilon<1$. Also, we show
that the spectral trace of the operator $V_{\phi}$ always equals
$1$.}
 \footnote[0]{


 This research was partially supported by NAS of Ukraine, Grant  \# 0105U006289.}

  \setcounter{section}{1}
  {\bf 1. Introduction.}

  It is well known that the  Volterra integration operator
$V:\ f(x)\rightarrow \int\limits_0^xf(t)dt$ defined on  $L_2[0,1]$
is quasinilpotent, that is, $\sigma(V)=\{0\}$. Let $\phi\in C[0,1]$
such that $\phi(0)=0$. It was pointed out in  \cite{Lyubich1} and
\cite{Lyubich} that an operator $V_{\phi}$ defined by
\begin{equation}
V_{\phi}:\
f(x)\rightarrow\int\limits_0^{\phi(x)}f(t)dt\label{domanov1}
\end{equation}
is quasinilpotent on $C[0,1]$ whenever $\phi(x)\leqslant x$ for all
$x\in [0,1]$.

Let $\phi :  [0,1]\longrightarrow [0,1]$ be a measurable function
and let $V_{\phi}\ :L_p[0,1]\longrightarrow L_p[0,1]$ $(1\leqslant
p<\infty)$ be defined by \eqref{domanov1}. It was proved in
\cite{Tong}  and  \cite{Whitley} that $V_{\phi}$ is quasinilpotent
on $L_p[0,1]$ if and only if $\phi(x)\leqslant x$ for almost all
$x\in [0,1]$. It was noted in \cite{Whitley} and proved in
\cite{Zima} that the spectral radius of $V_{x^\alpha}$ ( defined on
$L_p[0,1]$  or $C[0,1]$) is $1-\alpha$ $(0<\alpha<1)$. The detailed
investigation of the spectrum of the operator $V_{x^\alpha}$  was
done in \cite{Domanov1}, where it was shown that the point spectrum
$\sigma_p(V_{x^\alpha})$ of $V_{x^\alpha}$ is simple and
$\sigma_p(V_{x^\alpha})=\{(1-\alpha)\alpha^{n-1}\}_{n=1}^\infty$.
The oscillation properties of the eigenfunctions of $V_{x^\alpha}$
also were investigated in \cite{Domanov1}.

The aim of this paper is to prove the following theorem.
\begin{theorem}\label{main1}
Let $\phi :  [0,1]\longrightarrow [0,1]$ be a  nondecreasing
continuous function such that $\phi(x)>x$ for all $x\in (0,1)$, and
$V_{\phi}$ be defined on $L_2[0,1]$ by \eqref{domanov1}. Set also
$\sigma_p(V_\phi)\setminus\{0\}=\{\lambda_n\}_{n=1}^\omega$ $(
1\le\omega\le\infty)$. Then:

 $(1)$ $\omega<\infty$ if and only if $\phi(0)>0$ and $\phi(1-\varepsilon)=1$ for some
$0<\varepsilon<1$;

 $(2)$  $\lim\limits_{\varepsilon\rightarrow
  0}\sum\limits_{|\lambda_n|>\varepsilon}\lambda_n=1$;

  $(3)$ $\sum\limits_{n=1}^\omega|\lambda_n|^{1+\varepsilon}<\infty$ for all
  $\varepsilon>0$.
\end{theorem}
The order of the material is as  follows.

In section 2 we recall some classical results in the theory of
trace-class operators, in the theory of Fredholm determinants and
in the theory of entire functions. In  section 3 we calculate the
Fredholm determinant $D_{V_{\phi}}(\lambda)$ of  the operator
$V_{\phi}$. In  section 4 we estimate the order of growth of
$D_{V_{\phi}}(\lambda)$ and prove Theorem \ref{main1}. It turns
out that the matrix trace of the operator $V_\phi$ is not defined,
but the spectral trace of $V_\phi$ does not depend on $\phi$ and
always equals  $1$.  This contrasts with the fact that
$\sigma_p(V_x)=\{\emptyset\}$.  We find also the spectral(=
matrix) traces of the $V_\phi^2$ and $V_\phi^3$. In  section 5 we
assume that $\phi : [0,1]\longrightarrow [0,1]$ is a strictly
increasing continuous function such that $card\{x:\
\phi(x)=x\}<\infty$ and describe the spectrum of $V_\phi$. Then we
consider  $V_\phi$ defined on the space $L_p[0,1]$.

  \setcounter{section}{2}
  {\bf 2. Preliminaries.} Here we recall some facts about trace class operators,  Fredholm determinants
  and entire functions.

 {\bf2.1.}
  Let $K$ be a compact operator defined on an infinite dimensional Hilbert space $\mathfrak H$. Let
  $s_n(K)$ ($n\geq 1$)  be the eigenvalues of $KK^*$.
  The operator $K$ is said to be \textit{of  class} $\mathbf{S_p}$ if $\sum\limits_{n=1}^\infty s_n(K)^p<\infty$.
  The trace $\mathbf{tr}K$ of  an operator $K\in\mathbf{S_1}$   is defined as its \textit{matrix trace}:
  $
  \mathbf{tr}K=\sum\limits_{n=1}^\infty(Ke_n,e_n),
  $
where $\{e_n\}_{n=1}^\infty$ is some orthonormal basis. It is
known that $\mathbf{tr} K$ does not depend on the choice of
$\{e_n\}_{n=1}^\infty$ and the series
$\sum\limits_{n=1}^\infty(Ke_n,e_n)$ converges absolutely.
   The celebrated theorem of Lidskii (see \cite{Gohberg and Krein1}) says that the matrix trace of an
   operator $K\in\mathbf{S_1}$ is equal to its \textit{spectral trace}, which is defined as
   the sum of eigenvalues of $K$ (counted with the algebraic multiplicity):
\begin{equation}\label{Lidskii}
  \mathbf{tr}K=\sum\limits_{n=1}^\infty(Ke_n,e_n)=\sum\limits_{n=1}^\omega\lambda_n,\qquad \omega\le\infty.
\end{equation}

  Let $K$ be an integral operator:
  $(Kf)(x)=\int\limits_0^1k(x,t)f(t)dt$ on  $L_2[0,1]$.
  It is well known (see \cite{Gohberg and Krein1}) that if $k(x,t)$ is a continuous function on $[0,1]\times
  [0,1]$,
  then  $K\in\mathbf{S_1}$  and $\mathbf{tr}K$ is given by the  integral of its diagonal:
\begin{equation}\label{trace}
   \mathbf{tr}K=\int\limits_0^1k(t,t)dt.
\end{equation}

 {\bf 2.2.} Now let $k(x,t)$ be  a bounded function on
$[0,1]\times [0,1]$. By definition, put
\begin{equation}\label{fredholm2}
D_K(\lambda):=\sum\limits_{n=0}^\infty
\frac{(-1)^n}{n!}A_n\lambda^n,
\end{equation}
where  $A_0:=1$ and
\begin{equation}\label{fredholm}
\begin{split}
A_n:=\int\limits_0^1\dots\int\limits_0^1 K(t_1,\dots,t_n)dt_1\dots
dt_n, \\
 K(t_1,\dots,t_n):=\det
\begin{pmatrix}
k(t_1,t_1)&\dots  &k(t_1,t_n)\\
\vdots    &\vdots &\vdots \\
k(t_n,t_1)&\dots  &k(t_n,t_n)
\end{pmatrix}
\end{split}
\end{equation}
for $n\ge 1$. The function $D_K(\lambda)$ is called\textit{ the
Fredholm determinant} of $K$.   Recall (see
\cite{Goursat,Lovitt,Tricomi} ) that:
\begin{flalign}\label{fredholm1}
(\text{i}) & &
A_n=n!\int\limits_0^1\int\limits_{t_1}^1\int\limits_{t_2}^1\dots\int\limits_{t_{n-1}}^1
K(t_1,\dots,t_n)dt_n\dots dt_1,& & n\ge 1;& &
\end{flalign}
$($ii$)$ $D_K(\lambda)$ is an entire function of $\lambda$ of the
order $\rho\le 2$;\\
$($iii$)$ $D_K(\mu^*)=0$  if and only if $\lambda^*:=1/\mu^*\in
\sigma_p(K)$  and the multiplicity of $\mu^*$ as a root of the
Fredholm determinant of $K$ is equal to the algebraic multiplicity
of the eigenvalue $\lambda^*$.

{\bf 2.3.} From Hadamard's theorem (Th 1, p.26, \cite{Levin}) and
Lindel\"{o}f's theorem (Th 3, p.33, \cite{Levin}), we get the
following
\begin{theorem}\label{EntireFunction}
 Let $f(z)$ be an entire function of  order $\rho_f\le
1$ and  type $\sigma_f<\infty$. Let also $\{a_n\}_{n=1}^\omega$
$(\omega\le \infty)$ be all roots of  $f(z)$ and $f(0)=1$. Then

$(i)$ if $\rho_f=1$, $\sigma_f=0$ and
$\sum\limits_{n=1}^\omega\frac{1}{|a_n|}<\infty$, then
$\omega=\infty$, $f(z)=\prod\limits_{n=1}^\infty(1-\frac{z}{a_n})$

\qquad and $\sum\limits_{n=1}^\infty\frac{1}{a_n}=-f'(0)$;

$(ii)$ if $\rho_f<1$, then $
f(z)=\prod\limits_{n=1}^\omega(1-\frac{z}{a_n})$ and
$\sum\limits_{n=1}^\omega\frac{1}{a_n}=-f'(0)$;

$(iii)$  if $\rho_f=0$, then
$\sum\limits_{n=1}^\omega\frac{1}{|a_n|^\varepsilon}<\infty$ for
each $\varepsilon>0$;

$(iv)$ if $\rho_f=1$, $\sigma_f=0$ and
$\sum\limits_{n=1}^\infty\frac{1}{|a_n|}=\infty$,

 \qquad then\ \  $
f(z)=e^{az}\prod\limits_{n=1}^\infty
\left(1-\frac{z}{a_n}\right)e^{z/a_n}$ and\ \ 
 $\limsup\limits_{r\rightarrow\infty}\Big\vert
a+\sum\limits_{|a_n|<r}\frac{1}{a_n}\Big\vert=0$.

 \qquad In particular,\ \ 
$\limsup\limits_{r\rightarrow\infty}\left(
\sum\limits_{|a_n|<r}\frac{1}{a_n}\right)=-a=-f'(0)$.

$(v)$
$\sum\limits_{n=1}^\omega\frac{1}{|a_n|^{1+\varepsilon}}<\infty$
for each $\varepsilon>0$.
\end{theorem}

 \setcounter{section}{3}
 {\bf 3. The Fredholm determinant of the
operator $V_\phi$.} We begin with an auxiliary lemma.
\begin{lemma}\label{detK}
Let $A=(a_{ij})_{i,j=1}^n$ be an $n\times n$ matrix all of whose
elements are $0$ or $1$ and  $a_{ij}=1$ for $1\leq j\leq i\leq n$.
 Then
$$
\det A=\prod\limits_{i=2}^n(1-a_{i-1i})=\begin{cases} 1,&
a_{i-1i}=0\text{ for }2\le i\le n
,\\
0,&\text{ otherwise. }
\end{cases}
$$
\end{lemma}
\begin{proof}
The proof is trivial.
\end{proof}
\begin{theorem}\label{Freddet1}
Let $\phi :  [0,1]\longrightarrow [0,1]$ be a nondecreasing
continuous function such that $\phi(x)>x$ for all $x\in (0,1)$. Let
also $V_{\phi}$ be defined on $L_2[0,1]$ by \eqref{domanov1}. Then
\begin{equation}\label{detVphi}
D_{V_\phi}(\lambda)=
1+\sum\limits_{n=1}^\infty(-1)^n\lambda^n\int\limits_0^1
\int\limits_{\phi(t_1)}^1\dots\int\limits_{\phi(t_{n-1})}^1dt_n\dots
dt_1.
\end{equation}
\end{theorem}
\begin{proof}
It is clear that
$(V_{\phi}f)(x)=\int\limits_0^1k(x,t)f(t)dt=:(Kf)(x)$, where
$$
k(x,t)=\chi(\phi(x)-t)=
\begin{cases}
1,&\phi(x)\ge t;\\
0,&\phi(x)<t;
\end{cases}.
$$
 Assume that $0\le t_1\le t_2\le\dots \le t_n\le
1$. Then
 $k(t_i,t_j)=1$ for $1\le j\le i\le n$ and the matrix
$\left(k(t_i,t_j)\right)_{i,j=1}^n$ satisfies the assumptions of
Lemma \ref{detK}. Hence, $K(t_1,\dots,
t_n)=\prod\limits_{i=2}^n(1-k(t_{i-1},t_i))$. Further, using
\eqref{fredholm2},\eqref{fredholm}, and \eqref{fredholm1} we get
\begin{equation*}
A_n=n!\int\limits_0^1\int\limits_{t_1}^1\int\limits_{t_2}^1\dots\int\limits_{t_{n-1}}^1
\prod\limits_{i=2}^n(1-k(t_{i-1},t_i))dt_n\dots
dt_1\\
=n!\int\limits_{\Omega_n}1 dt_n\dots dt_1,
\end{equation*}
where
\begin{equation*}
 \begin{split}
\Omega_n:=\{(t_1,\dots,t_n): 0\le t_1\le\dots\le t_n\le
1, k(t_1,t_2)=\dots=k(t_{n-1},t_n)=0\}
\\
=\{(t_1,\dots,t_n):\ 0\le t_1\le\phi(t_1)\le t_2\le\phi(t_2)
\le \dots\le \phi(t_{n-1})\le t_n\le 1\}.
\end{split}
\end{equation*}
That is
$$
A_n=n! \int\limits_0^1
\int\limits_{\phi(t_1)}^1\dots\int\limits_{\phi(t_{n-1})}^1dt_n\dots
dt_1,\ \ n\ge 1.
$$
This completes the proof.
\end{proof}
 \setcounter{section}{4}
  {\bf 4. The spectrum of the operator $V_{\phi}$.}

The following Proposition immediately follows from Theorem
\ref{Freddet1}.
\begin{proposition}
Let $\phi :  [0,1]\longrightarrow [0,1]$ be a nondecreasing
continuous function  such that  $\phi(x)> x$ for all $x\in (0,1)$.
Then $\sigma_p(V_\phi)\cap \mathbb{R}_-=\{\emptyset\}$.
\end{proposition}
\begin{lemma}\label{auxLemma4}
Suppose $\phi :  [0,1]\longrightarrow [0,1]$ is a nondecreasing
continuous function and  $\phi(x)>x$ for  $x\in (0,1)$; then the
following conditions are equivalent:

$($i$)$ $\phi(0)>0$ and $\phi(1-\varepsilon)=1$ for some
$0<\varepsilon<1$;

$($ii$)$ there exists a unique $N=N(\phi)\in\{2,3,\dots\}$ such
that $\phi^N(x):=\phi(\phi(\dots \phi(x)))=1$ for all  $x\in
[0,1]$ and $\phi^{N-1}(x_0)\ne 1$ for some $x_0\in[0,1)$.
\end{lemma}
\begin{proof}
The proof is left to the reader.
\end{proof}
\begin{theorem}\label{Th4.7}
Let $\phi :  [0,1]\longrightarrow [0,1]$ be a nondecreasing
continuous function such that  $\phi(x)>x$ for all $x\in (0,1)$.
Suppose also that $\phi(0)>0$, $\phi(1-\varepsilon)=1$ for some
$0<\varepsilon<1$, and $N=N(\phi)$ is determined
 by Lemma \ref{auxLemma4}
$($ii$)$. Then

$(1)$ $\sigma_p(V_\phi)\setminus\{0\}$ is a finite set. Moreover,
$\sigma_p(V_\phi)=\{0\}\bigcup(\lambda_1,\dots,\lambda_N)$, where
$\lambda_n\ne 0$;

$(2)$ $\sum\limits_{n=1}^N\lambda_n=1$.
\end{theorem}
\begin{proof}
It is easily shown that $0\in \sigma_p(V_\phi)$. Using Theorem
\ref{Freddet1}, we get
$D_{V_\phi}(\lambda)=1+\sum\limits_{n=1}^\infty A_n\lambda^n$,
where $ A_n=(-1)^n\int\limits_0^1\int\limits_{\phi(t_1)}^1\dots
\int\limits_{\phi(t_{n-1})}^1dt_n\dots dt_1$.  It is easily shown
that $\phi^{n-1}(t_1)\le t_n\le 1$. Since $\phi^n(x)=1$ for $n\ge
N$, it follows that $A_n=0$ for $n\ge N+1$. Therefore
$D_{V_\phi}(\lambda)$ is a polynomial of degree $N$ and $($1$)$ is
proved.  Further note that
$D_{V_\phi}(\lambda)=\prod\limits_{n=1}^N(1-\frac{\lambda}{a_n})$.
Thus
$\sum\limits_{n=1}^N\lambda_n=\sum\limits_{n=1}^N\frac{1}{a_n}
=-A_1=1$.
\end{proof}
 Let $\alpha_i(x),
\beta_i(x)\in C[0,1]$ ($1\leq i\leq n$). By definition, put
$$
\left\{
     \genfrac{}{}{0pt}{}{\alpha_1}{\beta_1}
     ,\dots,
     \genfrac{}{}{0pt}{}{\alpha_n}{\beta_n}
     \right\}:=
     \int\limits_{\beta_1(x)}^{\alpha_1(x)}\int\limits_{\beta_2(x_1)}^{\alpha_2(x_1)}\dots
     \int\limits_{\beta_n(x_{n-1})}^{\alpha_n(x_{n-1})}dx_n\dots
     dx_1.
$$
So $\left\{
     \genfrac{}{}{0pt}{}{\alpha_1}{\beta_1}
     ,\dots,
     \genfrac{}{}{0pt}{}{\alpha_n}{\beta_n}
     \right\}$ is a function of $x$.
 It is clear that
$$
\left\{
     \genfrac{}{}{0pt}{}{\alpha_1}{\beta_1},
     \dots,
     \genfrac{}{}{0pt}{}{\alpha_{i-1}}{\beta_{i-1}},
     \genfrac{}{}{0pt}{}{\alpha_i}{\beta_i},
     \genfrac{}{}{0pt}{}{\alpha_{i+1}}{\beta_{i+1}},
     \dots,
     \genfrac{}{}{0pt}{}{\alpha_n}{\beta_n}
      \right\}+
       \left\{
     \genfrac{}{}{0pt}{}{\alpha_1}{\beta_1},
     \dots,
     \genfrac{}{}{0pt}{}{\alpha_{i-1}}{\beta_{i-1}},
     \genfrac{}{}{0pt}{}{\gamma_i}{\alpha_i},
     \genfrac{}{}{0pt}{}{\alpha_{i+1}}{\beta_{i+1}},
     \dots,
     \genfrac{}{}{0pt}{}{\alpha_n}{\beta_n}
      \right\}
$$
\begin{equation}\label{additivity}
\begin{split}
  =
   \left\{
     \genfrac{}{}{0pt}{}{\alpha_1}{\beta_1},
     \dots,
     \genfrac{}{}{0pt}{}{\alpha_{i-1}}{\beta_{i-1}},
     \genfrac{}{}{0pt}{}{\alpha_i}{\beta_i}+
     \genfrac{}{}{0pt}{}{\gamma_i}{\alpha_i},
     \genfrac{}{}{0pt}{}{\alpha_{i+1}}{\beta_{i+1}},
     \dots,
     \genfrac{}{}{0pt}{}{\alpha_n}{\beta_n}
      \right\}\\
     =
        \left\{
     \genfrac{}{}{0pt}{}{\alpha_1}{\beta_1},
     \dots,
     \genfrac{}{}{0pt}{}{\alpha_{i-1}}{\beta_{i-1}},
         \genfrac{}{}{0pt}{}{\gamma_i}{\beta_i},
     \genfrac{}{}{0pt}{}{\alpha_{i+1}}{\beta_{i+1}},
     \dots,
     \genfrac{}{}{0pt}{}{\alpha_n}{\beta_n}
      \right\}.
\end{split}
\end{equation}
The following lemmas are needed.
  \begin{lemma}\label{auxLemma1}
  Let $0<\varepsilon_1<\varepsilon_2<1$ and
  $$
  \psi(x)=\begin{cases}
  \psi_1(x),& x\in [0,\varepsilon_1];\\
  \psi_2(x),& x\in [\varepsilon_1,\varepsilon_2];\\
  \psi_3(x),& x\in [\varepsilon_2,1];
  \end{cases}
  $$
  be a strictly increasing continuous function such that
  $\psi(\varepsilon_1)=\varepsilon_1$ and $\psi(\varepsilon_2)=\varepsilon_2$. Let also
  $a_0=b_0=c_0=1$  and
  $a_k$, $b_k$, $c_k$, $d_k$ ($k=1,2\dots$) be  $k$-multiple integrals defined by
 \begin{eqnarray*}
  a_k:=\left\{
  \genfrac{}{}{0pt}{}{\varepsilon_1}{0},
  \genfrac{}{}{0pt}{}{\psi_1}{0},
     \dots,
  \genfrac{}{}{0pt}{}{\psi_1}{0}
     \right\},&\quad &
 b_k:=\left\{
  \genfrac{}{}{0pt}{}{\varepsilon_2}{\varepsilon_1},
  \genfrac{}{}{0pt}{}{\psi_2}{\varepsilon_1},
     \dots,
  \genfrac{}{}{0pt}{}{\psi_2}{\varepsilon_1}
     \right\},\quad
\\
 c_k:=\left\{
     \genfrac{}{}{0pt}{}{1}{\varepsilon_2},
     \genfrac{}{}{0pt}{}{\psi_3}{\varepsilon_2},
     \dots,
     \genfrac{}{}{0pt}{}{\psi_3}{\varepsilon_2}
     \right\},&\quad &
     d_{k}:=\left\{
     \genfrac{}{}{0pt}{}{1}{0},
     \genfrac{}{}{0pt}{}{\psi}{0},
     \dots,
     \genfrac{}{}{0pt}{}{\psi}{0}
     \right\}.
\end{eqnarray*}
 Then
  $$
  d_n=\sum\limits_{k=0}^n c_k\sum\limits_{l=0}^{n-k}b_la_{n-k-l},\qquad n=1,2,\dots.
  $$
  \end{lemma}
  \begin{proof}
  Using \eqref{additivity}, we get
  \begin{equation*}
  \begin{split}
  d_n=\left\{
      \genfrac{}{}{0pt}{}{\varepsilon_1}{0}
         +
      \genfrac{}{}{0pt}{}{\varepsilon_2}{\varepsilon_1}
      +
    \genfrac{}{}{0pt}{}{1}{\varepsilon_2}
     ,
     \genfrac{}{}{0pt}{}{\psi}{0},
     \dots,
     \genfrac{}{}{0pt}{}{\psi}{0}
     \right\}
     &=
     \left\{
     \genfrac{}{}{0pt}{}{\varepsilon_1}{0},
     \genfrac{}{}{0pt}{}{\psi_1}{0},
     \dots,
     \genfrac{}{}{0pt}{}{\psi_1}{0}
     \right\}\\ +
     \left\{
     \genfrac{}{}{0pt}{}{\varepsilon_2}{\varepsilon_1},
     \genfrac{}{}{0pt}{}{\varepsilon_1}{0}+
     \genfrac{}{}{0pt}{}{\psi_2}{\varepsilon_1},
     \genfrac{}{}{0pt}{}{\psi}{0},
     \dots,
     \genfrac{}{}{0pt}{}{\psi}{0}
     \right\}
     &+\left\{
     \genfrac{}{}{0pt}{}{1}{\varepsilon_2},
     \genfrac{}{}{0pt}{}{\varepsilon_1}{0}+
     \genfrac{}{}{0pt}{}{\varepsilon_2}{\varepsilon_1}+
     \genfrac{}{}{0pt}{}{\psi_3}{\varepsilon_2},
     \genfrac{}{}{0pt}{}{\psi}{0},
     \dots,
     \genfrac{}{}{0pt}{}{\psi}{0}
     \right\}\\ &=
    : K_n+L_n+M_n.
  \end{split}
  \end{equation*}
  By definition  $K_{n}=a_n$. Further, again using \eqref{additivity}, we get
  \begin{equation*}
  \begin{split}
  L_n=\ \left\{
     \genfrac{}{}{0pt}{}{\varepsilon_2}{\varepsilon_1},
     \genfrac{}{}{0pt}{}{\varepsilon_1}{0},
     \genfrac{}{}{0pt}{}{\psi_1}{0},
     \dots,
     \genfrac{}{}{0pt}{}{\psi_1}{0}
     \right\}+
  \left\{
     \genfrac{}{}{0pt}{}{\varepsilon_2}{\varepsilon_1},
     \genfrac{}{}{0pt}{}{\psi_2}{\varepsilon_1},
     \genfrac{}{}{0pt}{}{\varepsilon_1}{0}+
     \genfrac{}{}{0pt}{}{\psi_2}{\varepsilon_1},
     \genfrac{}{}{0pt}{}{\psi}{0},
     \dots,
     \genfrac{}{}{0pt}{}{\psi}{0}
     \right\}
\\
=\ b_1a_{n-1}+
     \left\{
     \genfrac{}{}{0pt}{}{\varepsilon_2}{\varepsilon_1},
     \genfrac{}{}{0pt}{}{\psi_2}{\varepsilon_1},
     \genfrac{}{}{0pt}{}{\varepsilon_1}{0},
     \genfrac{}{}{0pt}{}{\psi_1}{0},
     \dots,
     \genfrac{}{}{0pt}{}{\psi_1}{0}
     \right\}+
\left\{
     \genfrac{}{}{0pt}{}{\varepsilon_2}{\varepsilon_1},
     \genfrac{}{}{0pt}{}{\psi_2}{\varepsilon_1},
     \genfrac{}{}{0pt}{}{\psi_2}{\varepsilon_1},
     \genfrac{}{}{0pt}{}{\varepsilon_1}{0}+
     \genfrac{}{}{0pt}{}{\psi_2}{\varepsilon_1},
     \genfrac{}{}{0pt}{}{\psi}{0},
     \dots,
     \genfrac{}{}{0pt}{}{\psi}{0}
     \right\}
\\
 =\ b_1a_{n-1}+b_2a_{n-2}+\left\{
     \genfrac{}{}{0pt}{}{\varepsilon_2}{\varepsilon_1},
     \genfrac{}{}{0pt}{}{\psi_2}{\varepsilon_1},
     \genfrac{}{}{0pt}{}{\psi_2}{\varepsilon_1},
     \genfrac{}{}{0pt}{}{\psi_2}{\varepsilon_1},
     \genfrac{}{}{0pt}{}{\psi}{0},
     \dots,
     \genfrac{}{}{0pt}{}{\psi}{0}
     \right\}=\dots=\sum\limits_{k=1}^nb_ka_{n-k},
\end{split}
\end{equation*}
\begin{equation*}
\begin{split}
  M_n&=\ \left\{
     \genfrac{}{}{0pt}{}{1}{\varepsilon_2},
     \genfrac{}{}{0pt}{}{\varepsilon_1}{0},
     \genfrac{}{}{0pt}{}{\psi_1}{0},
     \dots,
     \genfrac{}{}{0pt}{}{\psi_1}{0}
     \right\}+
\left\{
     \genfrac{}{}{0pt}{}{1}{\varepsilon_2},
     \genfrac{}{}{0pt}{}{\varepsilon_2}{\varepsilon_1},
     \genfrac{}{}{0pt}{}{\psi_2}{0},
     \genfrac{}{}{0pt}{}{\psi}{0},
     \dots,
     \genfrac{}{}{0pt}{}{\psi}{0}
     \right\}\\
 &\qquad\qquad\qquad\qquad\qquad\quad+
 \left\{
     \genfrac{}{}{0pt}{}{1}{\varepsilon_2},
     \genfrac{}{}{0pt}{}{\psi_3}{\varepsilon_2},
     \genfrac{}{}{0pt}{}{\psi}{0},
     \dots,
     \genfrac{}{}{0pt}{}{\psi}{0}
 \right\}
  \\
&=\ c_1a_{n-1}+c_1L_{n-1}+ \left\{
     \genfrac{}{}{0pt}{}{1}{\varepsilon_2},
     \genfrac{}{}{0pt}{}{\psi_3}{\varepsilon_2},
      \genfrac{}{}{0pt}{}{\varepsilon_1}{0}+
     \genfrac{}{}{0pt}{}{\varepsilon_2}{\varepsilon_1}+
     \genfrac{}{}{0pt}{}{\psi_3}{\varepsilon_2},
      \genfrac{}{}{0pt}{}{\psi}{0},
     \dots,
     \genfrac{}{}{0pt}{}{\psi}{0}
     \right\}
\\
&=\ c_1a_{n-1}+c_1L_{n-1}+c_2a_{n-2}
 +c_2L_{n-2}  \\
  &\qquad\qquad\qquad\qquad\qquad\quad+\left\{
     \genfrac{}{}{0pt}{}{1}{\varepsilon_2},
     \genfrac{}{}{0pt}{}{\psi_3}{\varepsilon_2},
          \genfrac{}{}{0pt}{}{\psi_3}{\varepsilon_2},
      \genfrac{}{}{0pt}{}{\varepsilon_1}{0}+
     \genfrac{}{}{0pt}{}{\varepsilon_2}{\varepsilon_1}+
     \genfrac{}{}{0pt}{}{\psi_3}{\varepsilon_2},
      \genfrac{}{}{0pt}{}{\psi}{0},
     \dots,
     \genfrac{}{}{0pt}{}{\psi}{0}
 \right\} 
 \\
&=\dots=\
\sum\limits_{k=1}^nc_ka_{n-k}+\sum\limits_{k=1}^{n-1}c_kL_{n-k}=
 \sum\limits_{k=1}^nc_ka_{n-k}+\sum\limits_{k=1}^nc_k\sum\limits_{l=1}^{n-k}b_la_{n-k-l}.
\end{split}
\end{equation*}
  Finally, we obtain
  \begin{equation*}
\begin{split}
  d_n=K_n+L_n+M_n&=c_0a_n+\sum\limits_{k=1}^nb_ka_{n-k}+ \sum\limits_{k=1}^nc_ka_{n-k}+
  \sum\limits_{k=1}^nc_k\sum\limits_{l=1}^{n-k}b_la_{n-k-l}\\
&=\sum\limits_{k=0}^n c_k\sum\limits_{l=0}^{n-k}b_la_{n-k-l}.
 \end{split}
\end{equation*}
  \end{proof}
\begin{lemma}\label{auxlemma3}
Let $0<\varepsilon\leq 1/4$, $\beta>1$, and
  $$
  \psi_{\varepsilon,\beta}(x)=\begin{cases}
  x,& x\in [0,\varepsilon];\\
  \varepsilon+(1-2\varepsilon)^{1-\beta}(x-\varepsilon)^\beta,& x\in [\varepsilon,1-\varepsilon];\\
  x,& x\in [1-\varepsilon,1];
  \end{cases}
  $$
  Then
\begin{flalign}\label{eq1}
  &(1)\qquad 
     d_n:=\left\{
     \genfrac{}{}{0pt}{}{1}{0},
     \genfrac{}{}{0pt}{}{\psi_{\varepsilon,\beta}}{0},
     \dots,
     \genfrac{}{}{0pt}{}{\psi_{\varepsilon,\beta}}{0}
     \right\}=\frac{(2\varepsilon)^n}{n!}+\frac{(1-2\varepsilon)(2\varepsilon)^{n-1}}{(n-1)!}
\\
     &  \qquad\qquad\quad +\sum\limits_{l=2}^n
     \frac{(1-2\varepsilon)^l(2\varepsilon)^{n-l}}{(n-l)!(1+\beta)\dots(1+\beta+\dots\beta^{l-1})},\qquad
     n=1,2,\dots;\nonumber
\\ &(2)\qquad 
d_n<const(\varepsilon,\beta)\frac{(4\varepsilon)^n}{n!},\quad
n=1,2,\dots,\nonumber
\end{flalign}
where $const(\varepsilon,\beta)$ does not depend on $n$.
\end{lemma}
\begin{proof}
Substituting   $\psi_{\varepsilon,\beta}$ for  $\psi(x)$  in Lemma
\ref{auxLemma1} , we get \eqref{eq1}. Indeed, it is easily proved
that $a_l=c_l=\frac{\varepsilon^l}{l!}$ ($l=0,1,\dots n$). By
definition, put $\tilde{b}_1(x):=(1-2\varepsilon)^{1-\beta}(x-\varepsilon)^\beta$, \ 
$\psi_2(x):=\varepsilon+\tilde{b}_1(x)$, and 
$$
\tilde{b}_l(x):=\underbrace{\left\{
  \genfrac{}{}{0pt}{}{\psi_2}{\varepsilon},
     \dots,
  \genfrac{}{}{0pt}{}{\psi_2}{\varepsilon}
     \right\}}_{l},\quad l=2,3,\dots.
$$
Then
$\tilde{b}_{l+1}(x)=\int\limits_\varepsilon^{\psi_2(x)}\tilde{b}_l(t)dt$.
 It can easily be checked (by induction on $l$) that
$$
\tilde{b}_l(x)=
\frac{(1-2\varepsilon)^{l-\beta-\dots-\beta^l}(x-\varepsilon)^{\beta+\beta^2+\dots+\beta^l}}{(1+\beta)\dots(1+\beta+\dots+\beta^{l-1})},
\qquad l=2,3,\dots.
$$
Since $b_l=\tilde{b}_l(1-\varepsilon)$, we see that
\begin{equation}\label{1}
 b_0=1,\quad b_1=1-2\varepsilon,\quad
b_l=\frac{(1-2\varepsilon)^l}{(1+\beta)\dots(1+\beta+\dots+\beta^{l-1})},\quad
l=2,3,\dots
\end{equation}
 Using Lemma \ref{auxLemma1}, we get
\begin{equation}\label{2}
\begin{split}
 d_n=&\sum\limits_{k=0}^n
c_k\sum\limits_{l=0}^{n-k}b_la_{n-k-l}=
\sum\limits_{l=0}^nb_l\sum\limits_{k=0}^{n-l}c_ka_{n-k-l}=
\sum\limits_{l=0}^nb_l\sum\limits_{l=0}^{n-l}\frac{\varepsilon^k}{k!}\frac{\varepsilon^{n-k-l}}{(n-k-l)!}\\
=&\sum\limits_{l=0}^nb_l\frac{\varepsilon^{n-l}}{(n-l)!}\sum\limits_{l=0}^{n-l}\frac{(n-l)!}{k!(n-l-k)!}=
\sum\limits_{l=0}^nb_l\frac{(2\varepsilon)^{n-l}}{(n-l)!} \qquad
n=1,2,\dots.
\end{split}
\end{equation}
Substituting \eqref{1} for $b_l$ in \eqref{2} we get \eqref{eq1}.

{\bf (2)} Taking into account the inequality of arithmetic and
geometric means, we obtain
\begin{equation}\label{arifmandgeom}
 (1+\beta)\dots(1+\beta+\dots\beta^{l-1})\ge
2\beta^{1/2}3\beta^{2/2}\dots l\beta^{(l-1)/2}=l!\beta^{(l-1)l/4}.
\end{equation}
 Hence,
$$
b_l\le
\frac{(1-2\varepsilon)^l}{l!}\left(\frac{1}{\beta^{1/4}}\right)^{l^2-l}<\frac{(1-2\varepsilon)^l}{l!}.
$$
Let $N$ be a number such that
$\left(\frac{1}{\beta^{1/4}}\right)^{l^2-l}<\left(\frac{2\varepsilon}{1-2\varepsilon}\right)^l$
for $l>N$ (for example,
$N=[4\log_\beta\left(\frac{1}{2\varepsilon}-1\right)]+2$). Then
$b_l<\frac{(2\varepsilon)^l}{l!}$ for $l>N$. Using \eqref{2}, we
get for $n>N$
\begin{equation*}
\begin{split}
d_n&=\sum\limits_{l=0}^Nb_l\frac{(2\varepsilon)^{n-l}}{(n-l)!}+
\sum\limits_{l=N+1}^nb_l\frac{(2\varepsilon)^{n-l}}{(n-l)!}\\
&\le
\frac{(2\varepsilon)^n}{n!}\sum\limits_{l=0}^N\frac{n!}{l!(n-l)!}\left(\frac{1-2\varepsilon}{2\varepsilon}\right)^l+
\frac{(2\varepsilon)^n}{n!}\sum\limits_{l=N+1}^n\frac{n!}{l!(n-l)!}
\\
&\le
\frac{(2\varepsilon)^n}{n!}\left(\frac{1-2\varepsilon}{2\varepsilon}\right)^N\sum\limits_{l=0}^N\frac{n!}{l!(n-l)!}+
\frac{(2\varepsilon)^n}{n!}\sum\limits_{l=0}^n\frac{n!}{l!(n-l)!}\\
&\le
\frac{(4\varepsilon)^n}{n!}\left(\left(\frac{1-2\varepsilon}{2\varepsilon}\right)^N+1\right).
\end{split}
\end{equation*}
 This completes the proof.
\end{proof}
\begin{lemma}\label{auxLemma2}
Let $\beta>1$ and
$$
  \psi_\beta(x):=
  \begin{cases}
 2^{\beta-1}x^\beta,& x\in [0,1/2];\\
 2^{\beta-1}(x-1/2)^\beta+1/2,& x\in [1/2,1];
  \end{cases}=:
  \begin{cases}
  \psi_1(x),& x\in [0,1/2];\\
  \psi_2(x),& x\in [1/2,1];
  \end{cases}.
  $$
 Let also   $a_0=b_0=1$  and
  $a_k$, $b_k$,  and $d_k$ ($k=1,2\dots$) be  $k$-multiple integrals defined by
  \begin{equation*}
   \begin{split}
      a_k:=\left\{
  \genfrac{}{}{0pt}{}{1/2}{0},
  \genfrac{}{}{0pt}{}{\psi_1}{0},
     \dots,
  \genfrac{}{}{0pt}{}{\psi_1}{0}
     \right\},\quad
 &b_k:=\left\{
  \genfrac{}{}{0pt}{}{1}{1/2},
  \genfrac{}{}{0pt}{}{\psi_2}{1/2},
     \dots,
  \genfrac{}{}{0pt}{}{\psi_2}{1/2}
     \right\},\quad\\
&d_{k}:=\left\{
     \genfrac{}{}{0pt}{}{1}{0},
     \genfrac{}{}{0pt}{}{\psi_\beta}{0},
     \dots,
     \genfrac{}{}{0pt}{}{\psi_\beta}{0}
     \right\}.
     \end{split}
  \end{equation*}
 Then
 \begin{flalign}\label{eq3}
 (1)&\quad &
 d_n=&\sum\limits_{l=0}^nb_la_{n-l},\quad &n=1,2,\dots;&\quad&
\\
(2)&\qquad &
 d_n <&\frac{\beta^{\frac{-n^2/2+n}{4}}}{n!},\quad
 &n=1,2,\dots.&\quad&\nonumber
\end{flalign}
\end{lemma}
\begin{proof}
Substituting $1/2$ for $\varepsilon_1$ and $1$ for $\varepsilon_2$
in  Lemma \ref{auxLemma1}, we get \eqref{eq3}.
 Further, it is not
hard to prove that $a_1=b_1=1/2$ and\\ $
a_l=b_l=2^{-l}\left((\beta+1)\cdot\dots\cdot(\beta^{l-1}+\dots+1)\right)^{-1}$
for $l\ge 2$. Now, by \eqref{arifmandgeom},
$a_l\le\frac{2^{-l}}{\beta^{(l-1)l/4}l!}$ and
  \begin{equation*}
   \begin{split}
d_n&\le\sum\limits_{l=0}^n\frac{2^{-l}}{\beta^{(l-1)l/4}l!}\frac{2^{-n+l}}{\beta^{(n-l-1)(n-l)/4}(n-l)!}\\
&=
\frac{2^{-n}}{n!}\sum\limits_{l=0}^n\frac{n!}{l!(n-l)!}\beta^{\frac{-2(l-n/2)^2-n^2/2+n}{4}}<
\frac{\beta^{\frac{-n^2/2+n}{4}}}{n!}.
  \end{split}
 \end{equation*}
 \end{proof}
\begin{proposition}\label{growthofD}
Let $\phi :  [0,1]\longrightarrow [0,1]$ be a nondecreasing
continuous function.

$(1)$ If $\phi(x)>x$ for  $x\in (0,1)$ then the order of
$D_{V_\phi}(\lambda)$ does not exceed $1$, and

\quad if it equals $1$, $D_{V_\phi}(\lambda)$ is of minimal type;

$(2)$ if for some  $0<a<b<1$
  $$
  \phi(x)\geq f_{a,b}(x):=
  \begin{cases}
  \frac{b}{a}x,& x\in [0,a],\\
  \frac{1-b}{1-a}x+\frac{b-a}{1-a},&x\in [a,b],
  \end{cases}
  $$
  for  $x\in [0,1]$, then the order of
$D_{V_\phi}(\lambda)$ equals $0$.
\end{proposition}
\begin{proof}
$(1)$ Taking into account  Theorem \ref{Freddet1}, we obtain
$D_{V_\phi}(\lambda)=1+\sum\limits_{n=1}^\infty(-1)^nA_n\lambda^n$,
where
 $A_n=\left\{
     \genfrac{}{}{0pt}{}{1}{0},
     \genfrac{}{}{0pt}{}{1}{\phi},
     \dots,
     \genfrac{}{}{0pt}{}{1}{\phi}
     \right\}$.
Since $\phi(x)>x$ for each $0<\varepsilon<1/4$, it follows that
there exists $\beta>1$ such that
$\phi(x)\ge\psi_{\varepsilon,\beta}^{-1}(x)$. Using Lemma
\ref{auxlemma3}, we get
\begin{equation*}
 \begin{split}
     A_n=d_n&=\left\{
     \genfrac{}{}{0pt}{}{1}{0},
     \genfrac{}{}{0pt}{}{1}{\phi},
     \dots,
     \genfrac{}{}{0pt}{}{1}{\phi}
     \right\}<\left\{
     \genfrac{}{}{0pt}{}{1}{0},
     \genfrac{}{}{0pt}{}{1}{\psi_{\varepsilon,\beta}^{-1}},
     \dots,
     \genfrac{}{}{0pt}{}{1}{\psi_{\varepsilon,\beta}^{-1}}\right\}\\
&=
\left\{
     \genfrac{}{}{0pt}{}{1}{0},
     \genfrac{}{}{0pt}{}{\psi_{\varepsilon,\beta}}{0},
     \dots,
     \genfrac{}{}{0pt}{}{\psi_{\varepsilon,\beta}}{0}
     \right\}<const(\varepsilon,\beta)\frac{(4\varepsilon)^n}{n!}.
\end{split}
\end{equation*}
 Therefore the order of growth of $D_{V_\phi}(\lambda)$
does not exceed $1$. Assume that  the order of growth of
$D_{V_\phi}(\lambda)$ is equal to $1$. Then the type of
$D_{V_\phi}(\lambda)$ does not exceed $4\varepsilon$ for each
$\varepsilon<1/4$. Thus $D_{V_\phi}(\lambda)$ is of minimal type.

$(2)$ Since $\phi(x)\ge f_{a,b}(x)$ for some  $0<a<b<1$, it
follows that there exists $\beta>1$ such that
$\phi(x)\ge\psi_{\beta}^{-1}(x)$. Using Lemma \ref{auxLemma2}, we
get
\begin{equation*}
 \begin{split}
     A_n=d_n&=\left\{
     \genfrac{}{}{0pt}{}{1}{0},
     \genfrac{}{}{0pt}{}{1}{\phi},
     \dots,
     \genfrac{}{}{0pt}{}{1}{\phi}
     \right\}<\left\{
     \genfrac{}{}{0pt}{}{1}{0},
     \genfrac{}{}{0pt}{}{1}{\psi_{\beta}^{-1}},
     \dots,
     \genfrac{}{}{0pt}{}{1}{\psi_{\beta}^{-1}}\right\}\\
&=
\left\{
     \genfrac{}{}{0pt}{}{1}{0},
     \genfrac{}{}{0pt}{}{\psi_{\beta}}{0},
     \dots,
     \genfrac{}{}{0pt}{}{\psi_{\beta}}{0}
     \right\}<\frac{\beta^{\frac{-n^2/2+n}{4}}}{n!}.
\end{split}
\end{equation*}
 Therefore the order of growth of $D_{V_\phi}(\lambda)$
equals $0$.
\end{proof}
  \begin{theorem}\label{Theoreminfinite1}
Let $\phi :  [0,1]\longrightarrow [0,1]$ be a nondecreasing
continuous function such that  $\phi(x)>x$ for all $x\in (0,1)$.
Suppose that either $\phi(0)=0$ or $\phi(1-\varepsilon)\ne 1$ for
all $0<\varepsilon<1$.
 Then

  $(1)$ $\sigma_p(V_\phi)\setminus\{0\}:=(\lambda_1,\dots,\lambda_n,\dots)$ --- is an infinite set;

  $(2)$  $\lim\limits_{\varepsilon\rightarrow
  0}\sum\limits_{|\lambda_n|>\varepsilon}\lambda_n=1$;

  $(3)$ $\sum\limits_{n=1}^\omega|\lambda_n|^{1+\varepsilon}<\infty$ for all
  $\varepsilon>0$.
\end{theorem}
\begin{proof}
Using Theorem \ref{Freddet1}, we get
$D_{V_\phi}(\lambda)=1+\sum\limits_{n=1}^\infty(-1)^nA_n\lambda^n$,
where  $A_n=\left\{
     \genfrac{}{}{0pt}{}{1}{0},
     \genfrac{}{}{0pt}{}{1}{\phi},
     \dots,
     \genfrac{}{}{0pt}{}{1}{\phi}
     \right\}$.
It is easy to see that if either $\phi(0)=0$ or
$\phi(1-\varepsilon)\ne 1$ for all $0<\varepsilon<1$, then $A_n>0$
for $n\ge 0$. Therefore $D_{V_\phi}(\lambda)$ is not a polynomial
in $\lambda$. Now we apply Proposition \ref{growthofD} $(1)$.
Suppose that the order of $D_{V_\phi}(\lambda)$ is less than $1$;
then using Theorem \ref{EntireFunction} $($ii$)$, we get
$D_{V_\phi}(\lambda)=\prod\limits_{n=1}^\omega(1-\frac{\lambda}{a_n})$.
Since $D_{V_\phi}(\lambda)$ is not a polynomial, it follows that
$\omega=\infty$ and
$\sum\limits_{n=1}^\infty\lambda_n=\sum\limits_{n=1}^\infty\frac{1}{a_n}=-A_1/A_0=1$.
Now suppose that  the order of $D_{V_\phi}(\lambda)$ is  equal to
$1$; then $D_{V_\phi}(\lambda)$ is of minimal type. Thus the
spectrum of $V_\phi$ is an infinite set. Now, the application of
Theorem \ref{EntireFunction} $($i$)$, $($iv$)$ yields $(2)$.

$(3)$ follows from Theorem \ref{EntireFunction}.
\end{proof}

Now we are ready to prove the main result of the paper\\
$Proof\  of\ Theorem\ \ref{main1}$

$(1)$  follows from Theorem \ref{Th4.7} $(1)$ and Theorem
\ref{Theoreminfinite1} $(1)$.

$(2)$-$(3)$ follow from Theorem \ref{Th4.7} $(2)$ and Theorem
\ref{Theoreminfinite1} $(2)$-$(3)$.\qed

  \begin{theorem}\label{fab}
Let $\phi :  [0,1]\longrightarrow [0,1]$ be a nondecreasing
continuous function and
 for some $0<a<b<1$
  $$
  \phi(x)\geq
  \begin{cases}
  \frac{b}{a}x,& x\in [0,a],\\
  \frac{1-b}{1-a}x+\frac{b-a}{1-a},&x\in [a,b],
  \end{cases}
  $$
 for all $x\in [0,1]$.
Let also either $\phi(0)=0$ or $\phi(1-\varepsilon)\ne 1$ for all
$0<\varepsilon<1$.
 Then

$(1)$
$\sigma_p(V_\phi)\setminus\{0\}:=(\lambda_1,\dots,\lambda_n,\dots)$
--- is an infinite set;

$(2)$ $\sum\limits_{n=1}^\infty\lambda_n=1$;

$(3)$ $\sum\limits_{n=1}^\infty|\lambda_n|^\varepsilon<\infty$ for
all $\varepsilon>0$.
\end{theorem}
\begin{proof}
$(1)$  follows from Theorem \ref{Theoreminfinite1} $(1)$. By
Proposition \ref{growthofD} $(2)$, the order of
$D_{V_\phi}(\lambda)$ equals $0$. Thus $(2)$ and $(3)$ are implied
by (ii) and (iii) of Theorem \ref{EntireFunction}.
\end{proof}
\begin{remark}
$(i)$ Suppose $\phi(x)$ is a strictly increasing function and
$\phi(x)>x$ for all $x\in (0,1)$. Let also $\phi(x)\in C^1[0,1]$
and $(\phi'(x))^{-1/2}\in L_\infty[0,1]$. We claim that
$V_{\phi}\not\in\mathbf{S_1}$. Indeed, let $c:=\left(\int\limits_0^1(\phi '(s))^{1/2}ds\right)^{-1}$ and 
let $W_\phi$ and $T_\phi$
be linear operators defined on $L_2[0,1]$ by
$$
(W_\phi f)(x)=\int\limits_0^x(\phi '(t))^{1/2}f(t)dt,\quad (T_\phi
f)(x)=f(c\int\limits_0^x(\phi '(s))^{1/2}ds).
$$
It can easily be checked (see \cite{Domanov2}-\cite{Domanov3}) that
$T_\phi$ and $T_\phi^{-1}$ are bounded operators and
$cV_x=T_\phi^{-1}W_\phi T_\phi$. Hence, (see \cite{Gohberg and
Krein2}) $s_n(W_\phi)\ge
\|T_\phi\|^{-1}\|T_\phi^{-1}\|^{-1}s_n(cV_x)=\|T_\phi\|^{-1}\|T_\phi^{-1}\|^{-1}c\frac{2}{(2n-1)\pi}$.
Further,
$$
(V_\phi
V_\phi^*f)(x)=\int\limits_0^{\phi(x)}\int\limits_{\phi^{-1}(t)}^{1}f(s)dsdt=
\int\limits_0^x\phi '(t)\int\limits_t^{1}f(s)dsdt=(W_\phi
W_\phi^*f)(x).
$$
Thus $s_n(V_\phi)=s_n(W_\phi)\ge
\|T_\phi\|^{-1}\|T_\phi^{-1}\|^{-1}c\frac{2}{(2n-1)\pi}$. Hence,
$V_{\phi}\not\in\mathbf{S_1}$.

 $(ii)$ Since $V_{\phi}\not\in\mathbf{S_1}$, it
follows that the matrix trace of an operator $V_{\phi}$ is not
defined. Hence we cannot use  \eqref{Lidskii}-\eqref{trace} to
prove Theorem \ref{fab} (2). Nevertheless,
\eqref{Lidskii}-\eqref{trace} hold for $K=V_\phi$ and the
orthonormal basis $\{e_n\}_{n=1}^\infty$ defined
 by: $e_1\equiv 1$, $e_{2n}:=e^{2\pi i nx}$
and $e_{2n+1}:=e^{-2\pi i nx}$ ($n=1,2,\dots$). Indeed, since
$\sum\limits_{n=1}^\infty\frac{\sin nx}{n}=\frac{\pi-x}{2}$ for
$x\in (0,2\pi)$, it follows that
\begin{equation*}
\begin{split}
&\sum\limits_{n=0}^\infty(V_\phi
e_n,e_n)=
\int\limits_0^1\phi(x)dx\\
&+\sum\limits_{n=1}^\infty\left(\int\limits_0^1\frac{(e^{2\pi
in\phi(x)}-1)e^{-2\pi inx}}{2\pi
in}dx+\int\limits_0^1\frac{(e^{-2\pi in\phi(x)}-1)e^{2\pi
inx}}{-2\pi in}dx\right)\\
&=\int\limits_0^1\phi(x)dx+\sum\limits_{n=1}^\infty\int\limits_0^1\frac{\sin(2\pi
n(\phi(x)-x))}{\pi
n}dx\\
&=\int\limits_0^1\phi(x)dx+\int\limits_0^1\frac{1}{\pi}\frac{(\pi-2\pi(\phi(x)-x))}{2}dx=1.
\end{split}
\end{equation*}
Further, $\int\limits_0^1\chi(\phi(x)-x)dx=1$. Thus formulas
\eqref{Lidskii}-\eqref{trace} hold. This  contrasts with the fact
that $\sum\limits_{n=0}^\infty(V_x e_n,e_n)=\infty$.

$($iii$)$  Theorem \ref{main1} states that the spectral trace of an
operator $V_{\phi}$ always  equals $1$. This  also contrasts with
the fact that an  operator $V_x$ is quasinilpotent.
\end{remark}

To estimate the spectral radius $r(V_\phi)$ of the operator
$V_\phi$  we recall (see \cite{Zabreiko})  some results on
integral operators with nonnegative kernels. Let
$(Kf)(x)=\int\limits_0^1k(x,t)f(t)dt$ and $k(x,t)\ge 0$ for
$(x,t)\in [0,1]\times [0,1]$. If there exist  $\alpha>0$ and a
nonnegative function $f$ such that $(Kf)(x)\ge \alpha f(x)$ for
$x\in [0,1]$, then $r(K)\in \sigma_p(K)$ and $r(K)>\alpha$.
\begin{proposition}Let $\phi :  [0,1]\longrightarrow [0,1]$ be a strictly increasing
continuous function  such that  $\phi(x)\ge x$ for all $x\in
[0,1]$. Set also $\sigma_p(V_\phi)=\{\lambda_n\}_{n=1}^\omega$
$(\omega \le\infty )$. Then

$(1)$ $r(V_\phi)\ge \max\limits_{x\in [0,1]}(\phi(x)-x)$, \ \
$r(V_\phi)\in\sigma_p(V_\phi)$.

Let also $\phi(0)=0$. Then $\omega=\infty$ and

$(2)$
 $\sum\limits_{n=1}^\infty\lambda_n^2=2\int\limits_0^1
\phi(t)dt-1$;

$(3)$
  $\sum\limits_{n=1}^\infty\lambda_n^3=1-3\int\limits_0^1\phi(t)\phi^{-1}(t)dt$.
\end{proposition}
\begin{proof}
$(1)$ Let $f_a(x)=1-\chi(a-x)$, $a\in (0,1)$ then
$$
(V_\phi f_a)(x)=
\begin{cases}
0,& [0,\phi^{-1}(a)]\\
\phi(x)-a,& [\phi^{-1}(a),1]
\end{cases}
\ge (\phi(a)-a)f_a(x),
$$
and $(1)$ is proved.

$(2)$, $(3)$ It is easy to check that $\phi^{-1}(x)$ is well
defined and
$$
(V_\phi^2f)(x)=
\int\limits_0^1\chi(\phi^2(x)-t)(\phi(x)-\phi^{-1}(t))f(t)dt=:\int\limits_0^1k_2(x,t)f(t)dt.
$$
$$
(V_\phi^3f)(x)=
\int\limits_0^1\chi(\phi^3(x)-t)\int\limits_{\phi^{-2}(t)}^{\phi(x)}(\phi(s)-\phi^{-1}(t))dsf(t)dt=:\int\limits_0^1k_3(x,t)f(t)dt.
$$
Further, $k_2(x,t)$ and $k_3(x,t)$ are continuous functions on
$[0,1]\times [0,1]$. Hence, $V_\phi^2\in\mathbf{S_1}$ and
$V_\phi^3\in\mathbf{S_1}$. Now if we recall \eqref{trace}, we get
$$
\sum\limits_{n=1}^\infty\lambda_n^2=\int\limits_0^1k_2(t,t)dt=\int\limits_0^1(\phi(t)-\phi^{-1}(t))dt=2\int\limits_0^1
\phi(t)dt-1,
$$
$$
\sum\limits_{n=1}^\infty\lambda_n^3=
\int\limits_0^1k_3(t,t)dt=\int\limits_0^1
\int\limits_{\phi^{-2}(t)}^{\phi(t)}(\phi(s)-\phi^{-1}(t))ds
$$
$$=
\int\limits_0^1\left(\phi(t)\phi^2(t)-2\phi^{-1}(t)\phi(t)+\phi^{-1}(t)\phi^{-2}(t)\right)dt=
1-3\int\limits_0^1\phi(t)\phi^{-1}(t)dt.
$$
\end{proof}
\begin{example}
Let $\phi(x)=x^\alpha$ $(0<\alpha<1)$. It can be proved by direct
calculations that
$$
D_{V_{x^\alpha}}(\lambda)=
1+\sum\limits_{n=1}^\infty(-1)^n\lambda^n\int\limits_0^1
\int\limits_{t_1^\alpha}^1\dots\int\limits_{t_{n-1}^\alpha}^1dt_n\dots
dt_1
$$
$$
=1+\sum\limits_{n=1}^\infty(-1)^n\lambda^n\frac{\alpha^{n(n-1)/2}(1-\alpha)^n}{(1-\alpha)\dots
 (1-\alpha^n)}
=\prod\limits_{n=1}^\infty\left(1-\frac{\lambda}{(1-\alpha)\alpha^{n-1}}\right).
$$
Hence,
$\sigma_p(V_{x^\alpha})=\{(1-\alpha)\alpha^{n-1}\}_{n=1}^\infty$ and
each eigenvalue of $V_{x^\alpha}$ is of algebraic multiplicity one.
Further, $\sum\limits_{n=1}^\infty(1-\alpha)\alpha^{n-1}=1$ and
$\sum\limits_{n=1}^\infty\left((1-\alpha)\alpha^{n-1}\right)^\varepsilon=
\frac{(1-\alpha)^\varepsilon}{1-\alpha^\varepsilon}<\infty$ for each
$\varepsilon>0$.
\end{example}
 \setcounter{section}{5}
  {\bf 5. Some generalizations.}

 {\bf5.1.}
  The following Lemma  can be proved by direct calculations.
  \begin{lemma}\label{auxLemma5}
  Let $A$ be a compact operator defined on a Hilbert space $\mathfrak H$.
  Let also $\mathfrak H=\bigoplus\limits_{i=1}^k\mathfrak
  H_i$ and $A_i:=P_iA:\ \mathfrak H_i\rightarrow \mathfrak H_i $,
   where  $P_i$ be an orthoprojection in $\mathfrak H$
  onto $\mathfrak H_i$. Suppose
  that $\{\bigoplus\limits_{j=1}^i\mathfrak
  H_j\}_{i=1}^k$ is invariant for $A$; then
 $1/\lambda$ is an eigenvalue of $A$ of the algebraic multiplicity
$m\ge 1$ if and only if $1/\lambda$ is an eigenvalue of $A_i$ of
the algebraic multiplicity $m_i\ge 0$ and $\sum\limits_{i=1}^k
m_i=m$.
  \end{lemma}
  \begin{proof}
  The proof is omitted.
  \end{proof}
\begin{theorem}\label{Freddet2}
 Let $\phi :  [0,1]\longrightarrow [0,1]$  be a strictly increasing continuous
function. Let also   $\{ x:\ \phi(x)=x,\quad x\in
(0,1)\}=\{a_i\}_{i=1}^{k-1}$, where $0<a_1<\dots<a_{k-1}<1$ $(k\ge
2)$. By definition, put $a_0:=0$, $a_k:=1$, and
$$
\phi_i(x):=
(\phi(x(a_i-a_{i-1})+a_{i-1})-a_{i-1})/(a_i-a_{i-1}),\qquad 1\leq
i\leq k.
$$
$$
D_{V_{\phi_i}}(\lambda):=\begin{cases}
1+\sum\limits_{n=0}^\infty(-\lambda)^n
\left\{
     \genfrac{}{}{0pt}{}{1}{0},
     \genfrac{}{}{0pt}{}{1}{\phi_i},
     \dots,
     \genfrac{}{}{0pt}{}{1}{\phi_i}
     \right\}
%
%
,&\phi_i(x)>x\text{  for  }x\in (0,1);\\
1,& \phi_i(x)<x\text{  for  }x\in (0,1).
\end{cases}
$$
Then

$(1)$ $1/\lambda\in \sigma_p(V_\phi)$ if and only if
$\prod\limits_{i=1}^k D_{V_{\phi_i}}((a_i-a_{i-1})\lambda)=0$;

$(2)$ the algebraic multiplicity of the eigenvalue $1/\lambda$ is
equal to the  multiplicity of $\lambda$ as a root of the entire
function $\prod\limits_{i=1}^k
D_{V_{\phi_i}}((a_i-a_{i-1})\lambda)$.
\end{theorem}
\begin{proof}
By definition, put $\mathfrak H:=L_2[0,1]$, $\mathfrak
H_i:=L_2[a_{i-1},a_i]$ and
\begin{gather*}
P_i:\ f(x)\rightarrow
\begin{cases}
f(x),&x\in [a_{i-1},a_i];\\ 0,&x\not\in [a_{i-1},a_i];
\end{cases},\quad P_i:\ \mathfrak H\rightarrow \mathfrak H_i,\\
A:=V_\phi,\quad A_i:=P_iA\upharpoonleft_{\mathfrak H_i},\\
 T_i:\
\begin{cases} f(x),&x\in [a_{i-1},a_i];\\ 0,&x\not\in
[a_{i-1},a_i];
\end{cases}\rightarrow f((a_i-a_{i-1})x+a_{i-1}),\quad T_i:\ \mathfrak
H_i\rightarrow \mathfrak H.
\end{gather*}
It follows easily that $\bigoplus\limits_{j=1}^i\mathfrak
  H_j(=L_2[0,a_i])$ is invariant for $A$  and
\begin{gather*}
 A_i:\
\begin{cases}
f(x),&x\in [a_{i-1},a_i];\\ 0,&x\not\in [a_{i-1},a_i];
\end{cases}\rightarrow
\begin{cases}
\int\limits_{a_{n-1}}^{\phi(x)}f(t)dt,&x\in [a_{i-1},a_i];\\
0,&x\not\in [a_{i-1},a_i];
\end{cases},\\
T_i^{-1}:\ f(x)\rightarrow
 \begin{cases}
f(\frac{x-a_{i-1}}{a_i-a_{i-1}}),&x\in [a_{i-1},a_i];\\
0,&x\not\in [a_{i-1},a_i];
\end{cases},\quad T_i:\
\mathfrak H\rightarrow \mathfrak H_i,\\
T_iA_iT_i^{-1}=(a_i-a_{i-1})V_{\phi_i}.
\end{gather*}
The application of  Theorem \ref{Freddet1}  yields
$$
1/\lambda\in \sigma_p(A_i)\Leftrightarrow
1/\lambda\in\sigma_p((a_i-a_{i-1})V_{\phi_i})\Leftrightarrow
D_{V_{\phi_i}}((a_i-a_{i-1})\lambda)=0.
$$
The applying of Lemma \ref{auxLemma5} completes the proof.
\end{proof}
\begin{corollary}
Suppose  $\phi(x)$  satisfies the conditions of Theorem
\ref{Freddet2} and $mes\{x:\ \phi(x)\ge x,\ x\in
  [0,1]\}>0$. Set also
$\sigma_p(V_\phi)\setminus\{0\}=\{\lambda_n\}_{n=1}^\omega$
$(1\le\omega\le\infty)$. Then\\
$(1)$ $\omega<\infty$ if and only if $\phi(0)>0$,
$\phi(1-\varepsilon)=1$ for some $0<\varepsilon<1$
 and $\phi(x)>x$ for all $x\in (0,1)$;
\begin{flalign*}
 (2)&\quad &
  \lim\limits_{\varepsilon\rightarrow
  0}\sum\limits_{|\lambda_n|>\varepsilon}\lambda_n=mes\{x:\ \phi(x)\ge x,\ x\in
  [0,1]\}.&\quad &
\end{flalign*}
\end{corollary}
\begin{proof}
$(1)$ follows from Theorems \ref{Th4.7}, \ref{Theoreminfinite1},
\ref{Freddet2}.\\
$(2)$ By definition, put
\begin{gather*}
\Omega:=\{i:\ \phi(x)\ge x\ \text{ for }\ x\in [a_{i-1},a_i]\}=
\{i:\ \phi_i(x)\ge x\ \text{ for }\ x\in [0,1]\},\\
\sigma_p(V_{\phi_i}):=\{\lambda_{in}\}_{n=1}^{\omega_i},\qquad
1\le\omega\le\infty,\qquad i\in\Omega.
\end{gather*}
By Theorem \ref{Freddet2}
$$
\{\lambda_n\}_{n=1}^\omega=\sigma_p(V_\phi)=\bigcup\limits_{i\in\Omega}\sigma_p((a_i-a_{i-1})V_{\phi_i})=
\bigcup\limits_{i\in\Omega}(a_i-a_{i-1})\{\lambda_{in}\}_{n=1}^{\omega_i}.
$$
By Theorem \ref{Theoreminfinite1}
$$
 \lim\limits_{\varepsilon\rightarrow
  0}\sum\limits_{|\lambda_{in}|>\varepsilon}\lambda_{in}=1.
$$
Thus
\begin{equation*}
 \begin{split}
\lim\limits_{\varepsilon\rightarrow
  0}\sum\limits_{|\lambda_n|>\varepsilon}\lambda_n&=\sum\limits_{i\in\Omega}(a_i-a_{i-1})\lim\limits_{\varepsilon\rightarrow
  0}\sum\limits_{|\lambda_{in}|>\varepsilon}\lambda_{in}\\
&=\sum\limits_{i\in\Omega}(a_i-a_{i-1})=mes\{x:\
\phi(x)\ge x,\ x\in [0,1]\}.
 \end{split}
\end{equation*}
\end{proof}
\begin{remark}
It is interesting to note that  the  case of nonincreasing
function $\phi$ can be  more multifarious. In particular, if
$\phi(x)$ is a strictly decreasing continuous function such that
$\phi(0)=1$, $\phi(1)=0$ and $\phi(\phi(x))=x$ then $V_{\phi}$ is
a selfadjoint operator in $L_2[0,1]$. For example,
$\sigma_p(V_{1-x})=\{\frac{2(-1)^n}{(2n+1)\pi}\}_{n=1}^\infty$
 and $\sum\limits_{n=1}^\infty\frac{2(-1)^n}{(2n+1)\pi}=
 \frac{2}{\pi}\frac{\pi}{4}=\frac{1}{2}=mes\{x:
1-x\ge x\}$.
\end{remark}
{\bf 5.2.} In this subsection we consider an operator $V_\phi$
defined on $L_p[0,1]$ $(1\le p<\infty)$.

Let $A_i$ be a bounded operator defined on Banach space $X_i$
$(i=1,2)$. Recall that $A_1$ is said to be quasisimilar to $A_2$
if there exist deformations $K:\ X_1\rightarrow X_2$ and $L:\
X_2\rightarrow X_1$ (i.e. $\overline{\mathfrak{R}(K)}=X_2$, $ker
K=\{0\}$, $\overline{\mathfrak{R}(L)}=X_1$, $ker L=\{0\}$) such
that $A_1L=LA_2$ and $KA_1=A_2K$. It is clear that
$\sigma_p(A_1)=\sigma_p(A_2)$.
\begin{proposition}
 Let $\phi :  [0,1]\longrightarrow [0,1]$  be a strictly increasing continuous
function such that $\phi(0)=0$ and $\phi(1)=1$.
 Let $A_1$ denote an operator $V_\phi$ defined on $L_p[0,1]$
$(1\le p<\infty)$ and let $A_2$ denote an operator $V_\phi$
defined on $L_2[0,1]$. Then $A_1$ is quasisimilar to $A_2$, and
hence $\sigma_p(A_1)=\sigma_p(A_2)$.
\end{proposition}
\begin{proof}
By definition, put $K:=V_\phi:\ L_p[0,1]\rightarrow L_2[0,1]$,
$L:=V_\phi:\ L_2[0,1]\rightarrow L_p[0,1]$. It is clear that $K$
and $L$ are deformations and $A_1L=LA_2$, $KA_1=A_2K$.
\end{proof}

{\bf 5.3.} Now we consider the  operator
  $(V_{\phi,q,w}f)(x):=q(x)\int\limits_0^{\phi(x)}f(t)w(t)dt$ defined on $L_2[0,1]$.

  The proof of the following theorem is similar to the proof of
  Theorem \ref{Freddet1}.
\begin{theorem}\label{last}
Let $\phi :  [0,1]\longrightarrow [0,1]$ be a nondecreasing
continuous function such that $\phi(x)>x$ for all $x\in (0,1)$.
Let also $q(x),w(x)\in L_2[0,1]$. Then
\begin{equation*}
\begin{split}
&D_{V_{\phi,q,w}}(\lambda)=
\\
&=1+\sum\limits_{n=1}^\infty\ (-1)^n \lambda^n\int\limits_0^1
\int\limits_{\phi(t_1)}^1\dots\int\limits_{\phi(t_{n-1})}^1q(t_1)w(t_1)\dots
q(t_n)w(t_n)dt_n\dots dt_1.
\end{split}
\end{equation*}
 \end{theorem}
 \begin{corollary} Let the conditions of Theorem \ref{last} hold
 and $q(x)w(x)> 0$ for a.a. $x\in [0,1]$. Then $\sigma_p(V_{\phi,q,w})\setminus\{0\}$
is a finite set if and only if $\phi(0)>0$ and
$\phi(1-\varepsilon)=1$ for some $0<\varepsilon<1$.
 \end{corollary}
 {\bf Acknowledgments.} The author wishes to thank Professor J. Zem\'{a}nek
for setting the problem.

  Institute of Applied Mathematics and Mechanics

  Ukrainian National Academy of Sciences

  R. Luxemburg Str. 74

  83114  Donetsk

  Ukraine

 Mathematical Institute of the Academy of Sciences

 of the Czech Republic

 Zitna 25

 CZ - 115 67 Praha 1

 Czech Republic

 e-mail: domanovi@yahoo.com

   \end{document}